\documentclass[11pt]{article}
\usepackage[latin1]{inputenc}
\usepackage{amssymb}
\usepackage{amsmath}
\usepackage{amsfonts}
\usepackage{amsthm}
\usepackage{enumerate}

\numberwithin{equation}{section}

\newtheorem{theorem}{Theorem}[section]
\newtheorem{lemma}{Lemma}[section]
\newtheorem{proposition}{Proposition}[section]
\newtheorem{corollary}{Corollary}[section]

\newcommand\re[1]{(\ref{#1})}
\def\R{\mathbb{R}}
\def\Z{\mathbb{Z}}

\def\S{\mathcal{S}}

\def\N{\mathbb{N}}
\def\NN{\mathcal{N}}
\def\ZZ{\mathcal{Z}}
\def\F{\mathcal{F}}
\def\L{\mathcal{L}}
\def\T{\mathbb{T}}
\def\eps{\varepsilon}
\def\supp{\mathop{\rm supp}\nolimits}

\newcommand\cro[1]{\langle #1 \rangle}

\begin{document}
\begin{center}
\noindent {\Large \bf{Sharp ill-posedness and well-posedness results for the KdV-Burgers equation: the real line case}}
\end{center}
\vskip0.2cm
\begin{center}
\noindent
{\bf    Luc Molinet and St\'ephane Vento }\\
\end{center}
\vskip0.5cm \noindent {\bf Abstract.} { We complete the known results on the Cauchy problem in Sobolev spaces for the KdV-Burgers equation by
 proving that this equation is well-posed in $ H^{-1}(\R) $  with a solution-map that is analytic from $H^{-1}(\R) $ to  $C([0,T];H^{-1}(\R))$ whereas it is ill-posed
   in $ H^s(\R) $, as soon as $ s<-1 $, in the sense that the flow-map $u_0\mapsto u(t) $ cannot be continuous
    from $ H^s(\R) $ to even ${\cal D}'(\R) $ at any fixed $ t>0 $ small enough.  As far as we know, this is the first result of this type for a dispersive-dissipative equation.  The framework we develop  here should be  useful  to prove similar results  for other dispersive-dissipative models.}\vspace*{4mm} \\
\vskip0.2cm

\section{Introduction and main results}
The aim of this paper is to establish positive and negative optimal results on the local Cauchy problem in Sobolev spaces for the Korteweg-de Vries-Burgers (KdV-B) equation posed on the real line :
\begin{equation}\label{KdVB}
u_t +u_{xxx}-u_{xx}+u u_x =0
\end{equation}
where $ u=u(t,x) $ is a real valued function.\\
This  equation has been derived as an asymptotic  model for the propagation  of weakly nonlinear dispersive long waves in some physical contexts when dissipative effects occur (see \cite{OS}).
 It thus seems natural to compare the well-posedness results on the Cauchy problem
  for  the  KdV-B equation with the ones  for  the Korteweg-de-Vries (KdV) equation
 \begin{equation}\label{KDV}
 u_t+u_{xxx} +u u_x=0
 \end{equation}
 that  correspond to the case when dissipative
  effects are negligible and for the dissipative Burgers (dB) equation
 \begin{equation}\label{dB}
u_t-u_{xx} +u u_x=0
\end{equation}
  that corresponds to the case when dissipative effect are dominant.

  To make this comparison  more   transparent it is convenient to
   define     different notions of well-posedness (and consequently ill-posedness) related to the smoothness of the flow-map  (see in the same spirit \cite{KT}, \cite{Gerard}).

  Throughout this paper we shall say that  a Cauchy problem is  (locally)  $ C^0$-well-posed   in some  normed
  function space $ X $  if, for any initial data $ u_0\in X $, there exist a radius $ R>0 $, a  time
  $ T>0 $  and a unique solution $ u$,  belonging to some space-time function space
  continuously embedded in $ C([0,T];X) $, such that for any $ t\in [0,T] $ the map
   $ u_0\mapsto u(t) $ is continuous from the ball of $ X $ centered at $ u_0 $ with radius $ R $
    into $ X $.
    If   the map $ u_0\mapsto u(t) $ is of class $ C^k $, $k\in \N\cup\{\infty\} $, (resp. analytic) we will say that the Cauchy is $ C^k$-well-posed (resp. analytically well-posed).   Finally  a Cauchy problem will be  said to be $C^k $-ill-posed, $k\in  \N\cup\{\infty\} $,  if it is not $C^k $-well-posed.

For the KdV equation on the line  the situation is as follows: it is analytically  well-posed in $ H^{-3/4}(\R) $ (cf. \cite{KPV} and \cite{Guo} for the limit case) and  $ C^3 $-ill-posed below this index\footnote{See also \cite{CCT1} where it is proven that the solution-map is even not uniformly continuous on bounded sets below this index}  (cf. \cite{Bo1}).
  On the other hand the results for the dissipative Burgers equation
  are much clear. Indeed this equation is known to be analytically  well-posed in $ H^{s}(\R) $ for $ s\ge -1/2 $ (cf \cite{Di1} and \cite{Be} for the limit case)
  and $ C^0 $ ill-posed in $ H^s $ for $ s<-1/2 $ (cf.  \cite{Di1} ). At this stage it is interesting to notice that
   the critical Sobolev exponents obtained by  scaling considerations are respectively
    $ -3/2 $ for the KdV equation and $ -1/2 $ for the dissipative Burgers equation.  Hence for the KdV  equation there is an important  gap between this critical exponent and the best exponent obtained
     for  well-posedness.

Now, concerning  the KdV-Burgers equation,  Molinet and Ribaud  \cite{MR2} proved that this  equation is analytically  well-posed in $ H^s(\R) $ as soon as $ s>-1 $. They also established that the index $ -1 $ is critical for
 the  $C^2 $-well-posedness.  The surprising part of this result was that, according to the above results,  the $ C^\infty $ critical index $ s_c^\infty(KdVB)=-1 $
   was lower that the one of  the KdV equation $s_c^\infty(KdV)=-3/4 $    and also lower than the $ C^\infty $ index $ s_c^\infty(dB)=-1/2 $ of the dissipative Burgers
   equation.

 In this paper we want in some sense to complete this study by proving that the KdV-Burgers equation
 is analytically well-posed in $ H^{-1}(\R)$  and $ C^0 $-ill-posed in
  $H^s(\R) $ for $ s<-1 $ in the sense that the flow-map defined on $ H^{-1}(\R)$
   is not continuous for the topology inducted by $ H^s $, $s<-1 $, with values even in
    ${\cal D}' (\R)$.  It is worth emphasizing that the critical index $ s_c^0 =-1 $ is still far away from the critical index $ s_c=-3/2 $ given by the scaling symmetry of the KdV equation.
 We believe that this result strongly suggest that the KdV equation should also be $C^0$-ill-posed in $ H^s(\R) $
 for $ s<-1$.

To reach the critical Sobolev space $ H^{-1} (\R)$ we adapt the
      refinement of Bourgain's spaces that appeared in \cite{tataru} and
      \cite{Tao} to the framework developed in \cite{MR2}.  One of the main difficulty is related to the choice of the extension for negative times of the Duhamel operator (see the discussion in the beginning of Section \ref{sec-lin}). The approach we develop  here to overcome this difficulty should be useful to prove optimal  results
       for other dispersive-dissipative models. The ill-posedness result is due to a high to low frequency  cascade phenomena that was first observed in  \cite{BT} for a quadratic Schr\"odinger equation..

 At this stage it is worth noticing that, using the integrability theory, it was recently proved in \cite{KT} that the flow-map of KdV equation can be uniquely
  continuously extended in $H^{-1}(\T) $.  Therefore, on the torus,  KdV  is $ C^0 $-well-posed in $ H^{-1} $ if one takes as
   uniqueness class, the class of strong  limit in $C([0,T];H^{-1}(\T)) $ of smooth solutions. In the present work we use in a crucial way the global Kato smoothing effect that does not hold on the torus. However, in a forthcoming paper
    (\cite{MV2}) we will show how one can modify the approach developed here to prove that the  same results hold on the torus, i.e. analytic well-posedness in $ H^{-1}(\T)$  and $ C^0 $-ill-posedness in
  $H^s(\T) $ for $ s<-1 $. In view of the result of Kappeler and Topalov for KdV  it thus appears that, at least on the torus, even if the dissipation part of the KdV-Burgers equation
 (it is important to notice that
 the dissipative term $ -u_{xx} $ is of lower order than the dispersive one $ u_{xxx}$)
 allows to lower the $ C^\infty $ critical index with respect to the KdV equation,
       it does not  permit to improve the
      $ C^0$ critical index .

 % \subsection{Statement of the results}
  Our results can be summarized as follows:
\begin{theorem}\label{wellposed}
 The Cauchy problem associated to  (\ref{KdVB}) is  locally analytically  well-posed in $ H^{-1}(\R) $. 
Moreover,  at every point $ u_0 \in  H^{-1}(\R) $ there exist $ T=T(u_0)>0 $ and $R=R(u_0)>0 $ such that  the solution-map $ u_0\mapsto u  $ is analytic from the ball centered at $ u_0 $ with radius $R$ of $ H^{-1}(\R) $ into $ C([0,T];H^{-1}(\R)) $. Finally, the solution  $ u$ can be extended for all positive
  times  and belongs to $
 C(\R_+^*;H^\infty(\R)) $. 
\end{theorem}

\begin{theorem}\label{illposed}
The Cauchy problem associated to  (\ref{KdVB}) is ill-posed in $ H^s(\R) $ for $s <-1 $ in the following sense: there exists $ T>0 $ such that for any $ 0<t<T $,  the flow-map $ u_0\mapsto u(t) $ constructed in Theorem \ref{wellposed} is discontinuous at the origin from $ H^{-1}(\R) $  endowed with the topology inducted by $ H^s(\R) $ into ${\mathcal D}'(\R) $.
\end{theorem}
\noindent
{\bf Acknowlegements:}   L.M.  was partially supported by the ANR project
 "Equa-Disp".

        \section{Ill-posedness}
    The ill-posedness result can be viewed as an application of a  general result proved in
      \cite{BT}.
      Roughly speaking this general ill-posedness result   requires
       the two following ingredients:
   \begin{enumerate}
  \item  The equation is analytically well-posed until some index $ s_c^\infty $ with a solution-map that is also analytic.
  \item Below this index one iteration of the Picard scheme is not continuous. The discontinuity should be driven by  high frequency
   interactions
   that blow up in  frequencies of order
  least or equal to one.
   \end{enumerate}
   The first ingredient is given by Theorem \ref{wellposed} whereas the second one has  been derived
   in \cite{MR2} where the discontinuity of the second iteration of the Picard scheme
    in $ H^s(\R) $ and $ H^s(\T) $ for $ s<-1 $ is established.

   However, due to the nature of the equation, our result is a little better than the one
    given by the general theory developed in  \cite{BT}. Indeed, we will be able to prove the discontinuity of the flow-map $ u_0\mapsto u(t) $ for any fixed $ t>0 $ less than some $ T>0 $ and not only of the solution-map $ u_0\mapsto u $.
   Therefore for sake of completeness we will prove the result with hand here.

   %\subsection{Ill-posedness in $ H^{-1}(\R)$}
   Let us first recall the counter-example constructed in
   \cite{MR2} that we renormalize here in $ H^{-1}(\R)$.  We define the sequence of initial data $ \{\phi_N\}_{N\ge 1} $ by
    \begin{equation}
  \label{defphi}
  \hat{\phi}_N = N^{-1} \Bigl( \chi_{I_N}(\xi)+
    \chi_{I_N}(-\xi)\Bigr) \; ,
  \end{equation}
 where $ I_N =[N,N+2] $ and $ \hat{\phi}_N $ denotes the space Fourier transform of $ \phi_N$. \\
 Note that $ \| \phi_N \|_{H^{-1}(\R)} \sim 1 $ and $ \phi_N \to 0 $ in $ H^s(\R) $ for $ s<-1
 $.  This sequence yields a counter-example to the continuity of  the
 second iteration of the Picard Scheme  in $H^s(\R) $, $ s<-1$, that is given by
    $$
A_2(t,h,h) = \int_0^t S(t-t') \partial_x [S(t')h ]^2
   \, dt'
 $$
where $S$ is the semi-group associated to the linear part of (\ref{KdVB}) (see \re{trtr2}).
 Indeed, computing the space Fourier transform we get
 \begin{equation*}
 \begin{split}
{\cal F}_x(A_2(t,\phi_N,\phi_N))(\xi) & =  \int_\R
 e^{-t \xi^2} \, e^{it \xi^3} \,
\hat{\phi}_N(\xi_1) \hat{\phi}_N (\xi-\xi_1) \\
 &  \quad (i\xi) \, \int_0^t e^{-(\xi_1^2+
  (\xi-\xi_1)^2-\xi^2)t'} \,
 e^{i(\xi_1^3 + (\xi-\xi_1)^3 -\xi^3) t' } \, dt'\, d\xi_1  \\
 & =   (i\xi) \,e^{it \xi^3} \,e^{-t \xi^2} \int_\R
\hat{\phi}_N (\xi_1) \hat{\phi}_N (\xi-\xi_1)\,  \\
 & \hspace*{28mm}\frac{
e^{-(\xi_1^2+ (\xi-\xi_1)^2-\xi^2)t} \, e^{i 3 \xi \xi_1
(\xi-\xi_1)t}-1}{-2\xi_1 (\xi-\xi_1)+i  3\xi \xi_1 (\xi-\xi_1)} \,
 d\xi_1   \;,
 \end{split}
 \end{equation*}
so that
\arraycolsep2pt
\begin{eqnarray*}
\| A_2(t,\phi_N,\phi_N)\|_{H^{s}}^2
 & \ge & \int_{-1/2}^{1/2} (1+|\xi|^2)^{s} \, \left| {\cal
F}_x(A_2(t,\phi_N,\phi_N))(\xi)\right|^2 \, d\xi \\
 & = & N^{4} \int_{-1/2}^{1/2}
 (1+|\xi|^2)^{s} |\xi|^2 \,\\
  & & \hspace*{8mm} \Bigl| \int_{K_\xi}
 \frac{
e^{-(\xi_1^2+ (\xi-\xi_1)^2)t} \, e^{i 3 \xi \xi_1
(\xi-\xi_1)t}-e^{-\xi^2 \, t}}{-2\xi_1 (\xi-\xi_1)+i  3\xi \xi_1
(\xi-\xi_1)} \, d\xi_1 \Bigr|^2 \, d\xi \; ,
\end{eqnarray*}
where
$$
K_\xi =\{ \xi_1 \, / \, \xi -\xi_1 \in I_N , \, \xi_1 \in -I_N \}
 \cup \{ \xi_1 \, / \, \xi_1 \in I_N , \,\xi- \xi_1 \in -I_N \}
 \;.
$$
Note that for any $ \xi \in [-1/2,1/2] $, one has
 $ \mbox{mes}(K_\xi) \ge 1 $ and
 $$
\left\{ \begin{array}{rcl}
 3\xi \xi_1 (\xi-\xi_1) & \sim &  N^2 \\
 2 \xi_1 (\xi-\xi_1) & \sim &  N^2
 \end{array}
 \right.
 , \quad \forall \xi_1 \in K_\xi \, .
$$
Therefore, fixing  $ 0<t<1 $  we have
$$
{\cal{R}}e \, ( e^{-(\xi_1^2+ (\xi-\xi_1)^2)t} \, e^{i 3 \xi
\xi_1 (\xi-\xi_1)t}-e^{-\xi^2 \, t}) \le - e^{- t/4}
 + e^{-2 (N+2)^2 t } \; ,
$$
which leads for $ N=N(t) >0 $ large enough to
$$
\Bigl| \int_{K_\xi}
 \frac{
e^{-(\xi_1^2+ (\xi-\xi_1)^2)t} \, e^{i 3 \xi \xi_1
(\xi-\xi_1)t}-e^{-\xi^2 \, t}}{-2\xi_1 (\xi-\xi_1)+i  3\xi \xi_1
(\xi-\xi_1)} \, d\xi_1 \Bigr| \geq C  \frac{ e^{-
t/4}}{N^2 }
$$
and thus
\begin{equation}
\label{illposed1ineq4} \|A_{2}(t,\phi_N,\phi_N) \|_{H^s}^2 \geq C
e^{-t/4}\ge C_0
\end{equation}
for  some positive constant $ C_0>0 $. Since $ \phi_N\to 0 $  in $ H^s(\R) $, for $ s<-1 $, this ensures that, for any fixed $ t>0 $,  the map $ u_0\mapsto A_2(t,u_0,u_0) $ is not continuous at the origin from $ H^s(\R) $  into ${\cal D}' (\R)$. \\
 Now,
 we will use that  $ A_2(t,\phi_N,\phi_N) $ is of order at least one in $ H^{s}(\R)$  to prove that somehow $ A_2(t,\varepsilon\phi_N,\varepsilon\phi_N) $ is the main contribution to $ u(t,\varepsilon\phi_N) $ in $ H^s(\R) $ as soon as $ s<-1 $, $ \varepsilon >0 $
    is small and $ N $ is large enough. The discontinuity of $ u_0 \mapsto u(t) $
     will then follow from the one of $ u_0 \mapsto A_2(t,u_0,u_0) $.\\
     According to Theorem \ref{wellposed} there exist $ T>0 $ and $ \varepsilon_0>0 $ such that for any
      $|\varepsilon|\le \varepsilon_0$, any $\| h\|_{H^{-1}(\R)} \le 1$ and $0\le t\le T $,
       $$
       u(t,\varepsilon h) =\varepsilon S(t) h +  \sum_{k=2}^{+\infty} \varepsilon^k A_k(t,h^k)
     $$
     where $ h^k:=(h,\ldots,h) $,  $h^k\mapsto A_k(t,h^k) $ is a  $k$-linear  continuous map from $ H^{-1}(\R)^k $
      into  $ C([0,T];H^{-1}(\R)) $ and the series  converges absolutely in $ C([0,T];H^{-1}(\R)) $. In particular,
      $$
      u(t,\varepsilon \phi_N)-\varepsilon^2 A_2(t,\phi_N,\phi_N)= \varepsilon S(t) \phi_N
      + \sum_{k=3}^{+\infty} \varepsilon^k A_k(t,\phi_N^k) \; .
      $$
    On the other hand, $ \|S(t)\phi_N\|_{H^s(\R)}\le   \| \phi_N\|_{H^s(\R)} \sim N^{1+s}$
     and
              $$
\displaystyle  \Bigl\|\sum_{k=3}^\infty \varepsilon^k A_k
(t,\phi_N^k)\Bigr\|_{H^{-1}}
         \le \Bigl( \frac{\varepsilon}{\varepsilon_0}\Bigr)^3
         \sum_{k=3}^\infty \varepsilon_0^k\|  A_k
         (t,\phi_N)\|_{H^{-1}}\le C \varepsilon^3 \; .
         $$
Hence, for $ s<-1 $,
$$\sup_{t\in[0,T]}\Bigl\|u(t,\varepsilon \phi_N)-\varepsilon^2
A_2(t,\phi_N,\phi_N) \Bigr\|_{H^{s}(\R)}\le C
\varepsilon^3+O(N^{1+s})\; .$$ In view of  \re{illposed1ineq4}
this ensures that, fixing $ 0<t<1 $ and  taking  $ \varepsilon $ small enough and $ N $
large enough,
 $\varepsilon^2 A_2(t,\phi_N,\phi_N)$ is a ``good'' approximation of $
 u(t, \varepsilon \phi_N)$.  In particular, taking $
 \varepsilon \le C_0 C^{-1}/4 $ we get
 $$
\|u(t,\varepsilon \phi_N) \|_{H^{s}(\R)} \ge C_0\varepsilon^2/2
 +O(N^{1+s})  \; .
 $$
Since $ u(t,0)\equiv 0 $ and  $ \phi_N \to 0 $ in  $ H^s(\R) $
for $ s<-1 $ this leads to the discontinuity of the flow-map at
the origin by letting $ N $ tend  to infinity.
 It is worth noticing that since $ \phi_N \rightharpoonup  0 $ in
   $ H^{-1}(\R) $ we also get that  $ u_0 \mapsto u(t,u_0) $ is discontinuous from
  $ H^{-1}(\R) $ equipped with its weak topology with values even in ${\cal D}'(\R) $.
 \section{Resolution space}\label{sec-space}
In this section we introduce a few notation and we define our functional framework.

For $A,B>0$, $A\lesssim B$ means that there exists $c>0$ such that $A\leq cB$. When $c$ is a small constant we use $A\ll B$. We write $A\sim B$ to denote the statement that $A\lesssim B\lesssim A$.
For $u=u(t,x)\in\S'(\R^2)$, we denote by $\widehat{u}$ (or $\F_xu)$ its Fourier transform in space, and $\widetilde{u}$ (or $\F u$) the space-time Fourier transform of $u$. We consider the usual Lebesgue spaces $L^p$, $L^p_xL^q_t$ and abbreviate $L^p_xL^p_t$ as $L^p$. Let us define the Japanese bracket $\cro{x}=(1+|x|^2)^{1/2}$ so that the standard non-homogeneous Sobolev spaces are endowed with the norm $\|f\|_{H^s}=\|\cro{\nabla}^sf\|_{L^2}$.

We also need a Littlewood-Paley analysis. Let $\eta\in C^\infty_0(\R)$ be such that $\eta\geq 0$, $\supp \eta\subset [-2,2]$, $\eta\equiv 1$ on $[-1,1]$. We define next $\varphi(\xi)=\eta(\xi)-\eta(2\xi)$.
Any summations over capitalized variables such as $N,L$ are presumed to be dyadic, i.e. these variables range over numbers of the form $2^\ell$, $\ell\in \Z$. We set $\varphi_N(\xi)=\varphi(\xi/N)$ and define the operator
$P_N$ by $\F(P_Nu)=\varphi_N \widehat{u}$. We introduce $\psi_L(\tau,\xi)=\varphi_L(\tau-\xi^3)$ and for any $u\in\S'(\R^2)$,
$$\F_x(P_Nu(t))(\xi)=\varphi_N(\xi)\hat{u}(t,\xi),\quad \F(Q_Lu)(\tau,\xi)=\psi_L(\tau,\xi)\tilde{u}(\tau,\xi).$$
Roughly speaking, the operator $P_N$ localizes in the annulus $\{|\xi|\sim N\}$ whereas $Q_L$ localizes in the region $\{|\tau-\xi^3|\sim L\}$.

Furthermore we define more general projection $P_{\lesssim N}=\sum_{N_1\lesssim N}P_{N_1}$, $Q_{\gg L}=\sum_{L_1\gg L}Q_{L_1}$ etc.

Let $e^{-t\partial_{xxx}}$ be the propagator associated to the Airy equation and define the two parameters linear operator $W$ by
\begin{equation}
\F_x(W(t,t')\phi)(\xi)=\exp(it\xi^3-|t'|\xi^2)\hat{\phi}(\xi),\quad t\in\R.\label{trtr}
\end{equation}
The operator $W : t\mapsto W(t,t)  $ is clearly an extension to $ \R $ of the linear semi-group $ S(\cdot) $ associated with \re{KdVB} that is given by
\begin{equation}
\F_x(S(t)\phi)(\xi)=\exp(it\xi^3-t \xi^2)\hat{\phi}(\xi),\quad t\in\R_+.\label{trtr2}
\end{equation}
We will mainly work on the integral formulation of (\ref{KdVB}):
\begin{equation}\label{duhamel}
u(t) = S(t)u_0-\frac 12\int_0^tS(t-t')\partial_xu^2(t')dt',\quad t\in\R_+.
\end{equation}
Actually, to prove the local existence result, we will apply a fixed point argument to the following extension
 of (\ref{duhamel})  (See Section \ref{sec-lin} for some explanations on this choice).
\begin{eqnarray}\label{eq-int}u(t)& =&\eta(t)\Bigl[W(t)u_0-\frac 12 \chi_{\R_+}(t)\int_0^tW(t-t',t-t')\partial_xu^2(t')dt'
\nonumber \\
&&\hspace*{20mm}-\frac 12 \chi_{\R_-}(t)\int_0^tW(t-t',t+t')\partial_xu^2(t')dt'\Bigr] .
\label{eq-int}\end{eqnarray}
If $u$ solves (\ref{eq-int}) then $u$ is a solution of (\ref{duhamel}) on $[0, T]$, $T<1$.

In \cite{MR2}, the authors performed the iteration process in the space $X^{s,b}$ equipped with the norm
$$\|u\|_{X^{s,b}}=\|\cro{i(\tau-\xi^3)+\xi^2}^b\cro{\xi}^s\widetilde{u}\|_{L^2}$$
which take advantage of the mixed dispersive-dissipative part of the equation. In order to handle the endpoint index $s=-1$ without encountering logarithmic divergence, we will rather work in its Besov version $X^{s,b,q}$ (with $q=1$) defined as the weak closure of the test functions that are uniformly bounded by the norm
$$\|u\|_{X^{s,b,q}}=\Big(\sum_N\Big[\sum_L \cro{N}^{sq}\cro{L+N^2}^{bq}\|P_NQ_Lu\|_{L^2_{xt}}^q\Big]^{2/q}\Big)^{1/2}.$$
This Besov refinement, which usually provides suitable controls for nonlinear terms, is not sufficient here to get the desired bound especially in the \textit{high-high} regime, where the nonlinearity interacts two components of the solution $u$ with the same high frequency. To handle these divergences, inspired by \cite{Tao}, we introduce, for
$b\in\{\frac 12,-\frac 12\}$, the space $Y^{s,b}$ endowed with the norm
$$
\|u\|_{Y^{s,b}} =\Big(\sum_N [\cro{N}^s\|\F^{-1}[(i(\tau-\xi^3)+\xi^2+1)^{b+1/2}\varphi_N \widetilde{u}]\|_{L^1_tL^2_x}]^2\Big)^{1/2},
$$
so that
$$\|u\|_{Y^{-1,\frac 12}} \sim \Big(\sum_N[\cro{N}^{-1}\|(\partial_t+\partial_{xxx}-\partial_{xx}+I)P_Nu\|_{L^1_tL^2_x}]^2\Big)^{1/2}.$$

Next we form the resolution space $\S^{s}=X^{s,\frac 12,1}+Y^{s,\frac 12}$, and the "nonlinear space" $\NN^{s}=X^{s,-\frac 12,1}+Y^{s,-\frac 12}$ in the usual way:
$$\|u\|_{X+Y}=\inf\{\|u_1\|_X+\|u_2\|_Y: u_1\in X, u_2\in Y, u=u_1+u_2\}.$$

In the rest of this section, we study some basic properties of  the function space $\S^{-1}$.

\begin{lemma}\label{lem-Xinfty}
For any $\phi\in L^2$,
$$\Big(\sum_L[L^{1/2}\|Q_L(e^{-t\partial_{xxx}}\phi)\|_{L^2}]^2\Big)^{1/2}\lesssim \|\phi\|_{L^2}.$$
\end{lemma}
\begin{proof}
From Plancherel theorem, we have
$$\Big(\sum_L[L^{1/2}\|Q_L(e^{-t\partial_{xxx}}\phi)\|_{L^2}]^2\Big)^{1/2}\sim \||\tau-\xi^3|^{1/2}\F(e^{-t\partial_{xxx}}\phi)\|_{L^2}.$$
Moreover if we set $\eta_T(t)=\eta(t/T)$ for $T>0$, then
$$\F(\eta_T(t)e^{-t\partial_{xxx}}\phi)(\tau,\xi) = \widehat{\eta_T}(\tau-\xi^3)\widehat{\phi}(\xi).$$
Thus we obtain with the changes of variables $\tau-\xi^3\to \tau '$ and $T\tau'\to\sigma$ that
$$\||\tau-\xi^3|^{1/2}\F(\eta_T(t)e^{-t\partial_{xxx}}\phi)\|_{L^2}\lesssim \|\phi\|_{L^2}\||\tau'|^{1/2}T\widehat{\eta}(T\tau')\|_{L^2_{\tau'}}\lesssim \|\phi\|_{L^2}.$$
Taking the limit $T\to\infty$, this completes the proof.
\end{proof}

\begin{lemma}\label{lem-Sbounds}
\begin{enumerate}
\item For each dyadic $N$, we have
\begin{equation}\label{est-Y0}\|(\partial_t+\partial_{xxx})P_Nu\|_{L^1_tL^2_x}\lesssim \|P_Nu\|_{Y^{0,\frac 12}}.\end{equation}
\item For all $u\in \S^{-1}$,
\begin{equation}\label{est-L2S-1}\|u\|_{L^2_{xt}}\lesssim \|u\|_{\S^{-1}}.\end{equation}
\item For all $u\in \S^0$,
\begin{equation}\label{est-L2l2}\Big(\sum_L[L^{1/2}\|Q_Lu\|_{L^2}]^2\Big)^{1/2}\lesssim \|u\|_{\S^0}.\end{equation}
\end{enumerate}
\end{lemma}

\begin{proof}
\begin{enumerate}
\item From the definition of $Y^{0,\frac 12}$, the right-hand side of (\ref{est-Y0}) can be rewritten as
$$\|P_Nu\|_{Y^{0,\frac 12}}= \|(\partial_t+\partial_{xxx}-\partial_{xx}+I)P_Nu\|_{L^1_tL^2_x}.$$
Thus, by the triangle inequality, we reduce to show (\ref{est-Y0}) with $\partial_t+\partial_{xxx}$ replaced by $I-\partial_{xx}$.
Using Plancherel theorem as well as Young and H\"{o}lder inequalities, we get
\begin{align*}
&\|(I-\partial_{xx})P_Nu\|_{L^1_tL^2_x}\\  &\quad\lesssim \Big\|\F_t^{-1}\Big(\frac{\xi^2+1}{i(\tau-\xi^3)+\xi^2+1}(i(\tau-\xi^3)+\xi^2+1)\varphi_N\widetilde{u}\Big)\Big\|_{L^1_tL^2_\xi}.
\end{align*}
In the sequel, it will be convenient to write $\varphi_N $ for $ \varphi_{N/2} +\varphi_N +\varphi_{2N} $.
With this slight abuse of notation, we obtain
\begin{align*}
&\|(I-\partial_{xx})P_Nu\|_{L^1_tL^2_x}\\
&\quad\lesssim \Big\|\F_t^{-1}\Big(\frac{\varphi_N(\xi)(\xi^2+1)}{i(\tau-\xi^3)+\xi^2+1}\Big)\Big\|_{L^1_tL^\infty_\xi}\|(\partial_t+\partial_{xxx}-\partial_{xx}+I)P_Nu\|_{L^1_tL^2_x}.
\end{align*}
On the other hand, a direct computation yields
$$\Big|\F_t^{-1}\Big(\frac{\varphi_N(\xi)(\xi^2+1)}{i(\tau-\xi^3)+\xi^2+1}\Big)\Big| = C\varphi_N(\xi)(1+\xi^2)e^{-t(1+\xi^2)}\chi_{\R^+}(t)$$
so that
$$\Big\|\F_t^{-1}\Big(\frac{\varphi_N(\xi)(\xi^2+1)}{i(\tau-\xi^3)+\xi^2+1}\Big)\Big\|_{L^1_tL^\infty_\xi} \lesssim \|\cro{N}^2e^{-t\cro{N}^2}\chi_{\R^+}(t)\|_{L^1_t}\lesssim 1,$$
and the claim follows.
\item We show that for any fixed dyadic $N$, we have \begin{equation}\label{est-PNuS-1}\|P_Nu\|_{L^2}\lesssim \|P_Nu\|_{\S^{-1}}.\end{equation}
Estimate (\ref{est-L2S-1}) then follows after square-summing.
Observe that (\ref{est-PNuS-1}) follows immediately from the estimate $\cro{N}^{-1}\cro{L+N^2}^{1/2}\gtrsim 1$ if the right-hand side is replaced by $\|P_Nu\|_{X^{-1,\frac 12,1}}$, so it
suffices to prove (\ref{est-PNuS-1}) with $\|P_Nu\|_{Y^{-1,\frac 12}}$ in the right-hand side.
But applying again Young and H\"{o}lder's inequalities, this is easily verified:
\begin{align*}
\|P_Nu\|_{L^2} &= \Big\|\F_t^{-1}\Big(\frac{1}{i(\tau-\xi^3)+\xi^2+1}(i(\tau-\xi^3)+\xi^2+1)\varphi_N\widetilde{u}\Big)\Big\|_{L^2_{t\xi}}\\
&\lesssim \Big\|\F_t^{-1}\Big(\frac{\varphi_N(\xi)}{i(\tau-\xi^3)+\xi^2+1}\Big)\Big\|_{L^2_tL^\infty_x}\|P_Nu\|_{Y^{0,\frac 12}}\\
&\lesssim \|e^{-t\cro{N}^2}\chi_{\R^+}(t)\|_{L^2_t}\|P_Nu\|_{Y^{0,\frac 12}}\\
&\lesssim \cro{N}^{-1}\|P_Nu\|_{Y^{0,\frac 12}}\lesssim \|P_Nu\|_{Y^{-1,\frac 12}}.
\end{align*}

\item First it is clear from definitions that $\Big(\sum_L[L^{1/2}\|Q_Lu\|_{L^2}]^2\Big)^{1/2}\lesssim \|u\|_{X^{0,\frac 12,1}}$.

Setting now $v=(\partial_t+\partial_{xxx})u$, we see that $u$ can be rewritten as
$$u(t)=e^{-t\partial_{xxx}}u(0)+\int_0^te^{-(t-t')\partial_{xxx}}v(t')dt'.$$
By virtue of Lemma \ref{lem-Xinfty}, we have
$$\Big(\sum_L[L^{1/2}\|Q_Le^{-t\partial_{xxx}}u(0)\|_{L^2}]^2\Big)^{1/2}\lesssim \|u(0)\|_{L^2}\lesssim \|u\|_{L^\infty_tL^2_x}.$$
Moreover, we get as previously
\begin{equation}\label{est-LiL2}\|u\|_{L^\infty_tL^2_x}\lesssim \Big\|\F_t^{-1}\Big(\frac{1}{i(\tau-\xi^3)+\xi^2+1}\Big)\Big\|_{L^\infty_{t\xi}}\|u\|_{Y^{0,\frac 12}}
\lesssim \|u\|_{Y^{0,\frac 12}}.\end{equation}
Thanks to estimate (\ref{est-Y0}), it remains to show that
\begin{equation}\label{est-vL1L2}\Big(\sum_L\Big[L^{1/2}\Big\|Q_L\int_0^te^{-(t-t')\partial_{xxx}}v(t')dt'\Big\|_{L^2}\Big]^2\Big)^{1/2}\lesssim \|v\|_{L^1_tL^2_x}.\end{equation}
In order to prove this, we split the integral $\int_0^t = \int_{-\infty}^t-\int_{-\infty}^0$. By Lemma \ref{lem-Xinfty}, the contribution with integrand on $(-\infty,0)$ is bounded by
$$
\lesssim \Big\|\int_{-\infty}^0 e^{t'\partial_{xxx}}v(t')dt'\Big\|_{L^2_x}
\lesssim \|v\|_{L^1_tL^2_x}.
$$
For the last term, we reduce by Minkowski to show that
$$\Big(\sum_L[L^{1/2}\|Q_L(\chi_{t>t'}e^{-(t-t')\partial_{xxx}}v(t'))\|_{L^2_{tx}}]^2\Big)^{1/2}\lesssim \|v(t')\|_{L^2_x}.$$
This can be proved by a time-restriction argument. Indeed, for any $T>0$, we have
\begin{align*}
&\Big(\sum_L[L^{1/2}\|Q_L(\eta_T(t)\chi_{t>t'}e^{-(t-t')\partial_{xxx}}v(t'))\|_{L^2}]^2\Big)^{1/2}\\
&\quad \lesssim \||\tau|^{1/2} \widehat{v}(t')\F_t(\eta_T(t)\chi_{t>t'})(\tau)\|_{L^2}\\
&\quad\lesssim \|v(t')\|_{L^2}\||\tau|^{1/2}\F_t(\eta(t)\chi_{tT>t'})\|_{L^2}\\
&\quad\lesssim \|v(t')\|_{L^2}.
\end{align*}
We conclude by passing to the limit $T\to\infty$.
\end{enumerate}
\end{proof}

Now we state a general and classical result which ensures that our resolution space is well compatible with dispersive properties of the Airy equation.
Actually, it is a direct consequence of Lemma 4.1 in \cite{Tao} together with the fact that the resolution space $\S_0$ used by Tao to solve 4-KdV contains our space $\S^0$ thanks to estimate (\ref{est-Y0})

\begin{lemma}[Extension lemma]
Let $Z$ be a Banach space of functions on $\R\times\R$ with the property that
$$\|g(t)u(t,x)\|_Z \lesssim \|g\|_{L^\infty_t}\|u(t,x)\|_Z$$
holds for any $u\in Z$ and $g\in L^\infty_t(\R)$.
Let $T$ be a spacial linear operator for which one has the estimate
$$\|T(e^{-t\partial_{xxx}}P_N\phi)\|_Z\lesssim \|P_N\phi\|_{L^2}$$
for some dyadic $N$ and for all $\phi$.
Then one has the embedding
$$\|T(P_Nu)\|_Z\lesssim \|P_Nu\|_{\S^0}.$$
\end{lemma}

Combined with the unitary of the Airy group in $L^2$ and the sharp Kato smoothing effect
\begin{equation}\label{kato}\|\partial_xe^{-t\partial_{xxx}}\phi\|_{L^\infty_xL^2_t}\lesssim \|\phi\|_{L^2},\quad\forall\phi\in L^2,\end{equation}
we deduce the following result.
\begin{corollary} For any $u$, we have\footnotemark[1]
\begin{equation}\label{est-Lit}\|u\|_{L^\infty_tH^{-1}_x}\lesssim \|u\|_{\S^{-1}},\end{equation}
\begin{equation}\label{est-smooth}\|P_Nu\|_{L^\infty_xL^2_t}\lesssim N^{-1}\|P_Nu\|_{\S^{0}},\end{equation}
provided the right-hand side is finite. In particular, $\S^{-1} \hookrightarrow L^\infty_t H^{-1} $.
\end{corollary}
\footnotetext[1]{Note that (\ref{est-Lit}) can also be deduced from estimate (\ref{est-LiL2}).}

\section{Linear estimates}\label{sec-lin}
In this section we prove  linear estimates related to the operator $W$ as well as to the extension of the Duhamel operator introduced in \re{eq-int}.

 At this this stage let us give some explanations on our choice of this extension. Let us keep in mind that this extension has to be compatible with linear estimates in  both norms 
  $X^{s,1/2,1} $ and $ Y^{s,1/2} $. First, since $ X^{s,1/2,1} $ is a Besov in  time space we are not allowed to simply multiply the Duhamel term 
   by $ \chi_{\R^+}(t) $. Second,  in order to prove the desired linear estimate in $ Y^{s,1/2} $ the strategy is to use that the Duhamel term satisfies a forced KdV-Burgers equation. Unfortunately,  it turns out that the extension introduced in \cite{MR2}, that makes the calculus simple, does not satisfy such PDE for negative time. The new extension that we introduce in this work has the properties to satisfy some forced PDE related to KdV-Burgers for negative times (see \re{eses}) and to be compatible with linear 
     estimates in $X^{s,1/2,1} $. However the proof  is now a little more complicated even if it follows the same lines than the one of Propositions 2.3   in \cite{MR2}, see also Proposition 4.4, \cite{GuoWang}.
     
   The following lemma is a dyadic version of Proposition 2.1 in \cite{MR2}.
\begin{proposition}\label{lem-linhom}For all $\phi\in H^{-1}(\R)$, we have
\begin{equation}\label{est-lin}\|\eta(t) W(t)\phi\|_{\S^{-1}}\lesssim \|\phi\|_{H^{-1}}.\end{equation}
\end{proposition}

\begin{proof} We bound the left-hand side in (\ref{est-lin}) by the $X^{-1,\frac 12,1}$-norm of  $\eta(t)W(t)\phi$. After square-summing in $N$, we may reduce to prove
\begin{equation}\label{est-lindy}\sum_L\cro{L+N^2}^{1/2}\|P_NQ_L(\eta(t)W(t)\phi)\|_{L^2_{xt}}\lesssim \|P_N\phi\|_{L^2}\end{equation} for each dyadic $N$. Using Plancherel, we obtain
\begin{align*}
&\sum_L\cro{L+N^2}^{1/2}\|P_NQ_L(\eta(t)W(t)\phi)\|_{L^2_{xt}}\\ &\quad\lesssim \sum_L\cro{L+N^2}^{1/2}\|\varphi_N(\xi)\varphi_L(\tau)\widehat{\phi}(\xi)
\F_t(\eta(t)e^{-|t|\xi^2})(\tau)\|_{L^2_{\tau\xi}}\\
&\quad \lesssim \|P_N\phi\|_{L^2}\sum_L\cro{L+N^2}^{1/2}\|\varphi_N(\xi)P_L(\eta(t)e^{-|t|\xi^2})\|_{L^\infty_\xi L^2_t}.
\end{align*}
Hence it remains to show that
\begin{equation}\label{est-lin1}\sum_L\cro{L+N^2}^{1/2}\|\varphi_N(\xi)P_L(\eta(t)e^{-|t|\xi^2})\|_{L^\infty_\xi L^2_t}\lesssim 1.\end{equation}
We split the summand into $L\leq \cro{N}^2$ and $L\geq \cro{N}^2$. In the former case, we get by Bernstein
\begin{multline*}
\sum_{L\leq \cro{N}^2}\cro{L+N^2}^{1/2}\|\varphi_N(\xi)P_L(\eta(t)e^{-|t|\xi^2})\|_{L^\infty_\xi L^2_t} \\ 
\lesssim \sum_{L\leq \cro{N}^2}\cro{N}L^{1/2}\sup_{|\xi|\sim N}\|\eta(t)e^{-|t|\xi^2}\|_{ L^1_t}
\end{multline*}
Also, one can bound $\|\eta(t)e^{-|t|\xi^2}\|_{L^1}$ either by $\|\eta\|_{L^1}$ or by $\|e^{-|t|\xi^2}\|_{L^1_t}\sim |\xi|^{-2}$. It follows that
$$
\sum_{L\leq \cro{N}^2}\cro{L+N^2}^{1/2}\|\varphi_N(\xi)P_L(\eta(t)e^{-|t|\xi^2})\|_{L^\infty_\xi L^2_t} \lesssim \cro{N}^2 \min(1, N^{-2})\lesssim 1.
$$ 
Now we deal with the case $L\geq \cro{N}^2$. A standard paraproduct rearrangement allows us to write
\begin{align*}
P_L(\eta(t)e^{-|t|\xi^2}) &= P_L\Big(\sum_{M\gtrsim L}(P_M\eta(t) P_{\lesssim M}e^{-|t|\xi^2}+P_{\lesssim M}\eta(t) P_M e^{-|t|\xi^2}\Big)\\
& = P_L(I)+P_L(II).
\end{align*}
Using the Schur's test, the term $P_L(I)$ is directly bounded by
\begin{align*}
&\sum_{L\geq \cro{N}^2}\cro{L+N^2}^{1/2}\|\varphi_NP_L(I)\|_{L^\infty_\xi L^2_t}\\ &\quad \lesssim \sum_L L^{1/2}\sum_{M\gtrsim L}\|\varphi_NP_M\eta(t)\|_{L^\infty_\xi L^2_t}
\|\varphi_NP_{\lesssim M}e^{-|t|\xi^2}\|_{L^\infty_{\xi t}}\\
&\quad \lesssim \sum_M M^{1/2}\|P_M\eta\|_{L^2_t}\lesssim 1.
\end{align*}
Similarly for $P_L(II)$, we have
\begin{align*}
&\sum_{L\geq \cro{N}^2}\cro{L+N^2}^{1/2}\|\varphi_NP_L(II)\|_{L^\infty_\xi L^2_t}\\ 
&\quad \lesssim \sum_L L^{1/2}\sum_{M\gtrsim L}\|\varphi_NP_{\lesssim M}\eta(t)\|_{L^\infty_{\xi t}}
\|\varphi_NP_M e^{-|t|\xi^2}\|_{L^\infty_\xi L^2_t}\\
&\quad \lesssim \sum_M M^{1/2}\|\varphi_N P_M e^{-|t|\xi^2}\|_{L^\infty_\xi L^2_t}.
\end{align*}
Moreover, it is not too hard to check that if $|\xi|\sim N$, then $\|P_Me^{-|t|\xi^2}\|_{L^2_t}\lesssim \|P_Me^{-|t|N^2}\|_{L^2_t}$, thus
$$\sum_{L\geq \cro{N}^2}\cro{L+N^2}^{1/2}\|\varphi_NP_L(II)\|_{L^\infty_\xi L^2_t} \lesssim \sum_M M^{1/2}\|P_M e^{-|t|N^2}\|_{L^2_t}\lesssim 1,$$
where we used the fact that the Besov space $\dot{B}^{1/2}_{2,1}$ has a scaling invariance and $e^{-|t|}\in \dot{B}^{1/2}_{2,1}$.
\end{proof}

\begin{lemma}\label{lem-kxi}
For $w\in\S(\R^2)$, consider $k_\xi$ defined on $\R$ by
$$k_\xi(t) = \eta(t)\varphi_N(\xi)\int_\R\frac{e^{it\tau}e^{(t-|t|)\xi^2}-e^{-|t|\xi^2}}{i\tau+\xi^2}\widetilde{w}(\tau)d\tau.$$
Then, for all $\xi\in\R$, it holds 
$$\sum_L\cro{L+N^2}^{1/2}\|P_L k_\xi\|_{L^2_t} \lesssim \sum_L\cro{L+N^2}^{-1/2}\|\varphi_L(\tau)\varphi_N(\xi)\widetilde{w}\|_{L^2_\tau}.$$
\end{lemma}
\begin{proof}
Following \cite{MR2}, we rewrite $k_\xi$ as
\begin{align*}
k_\xi(t) &= \eta(t)e^{(t-|t|)\xi^2}\int_{|\tau|\leq 1}\frac{e^{it\tau}-1}{i\tau+\xi^2}\widetilde{w_N}(\tau)d\tau + \eta(t)\int_{|\tau|\leq 1}\frac{e^{(t-|t|)\xi^2}-e^{-|t|\xi^2}}{i\tau+\xi^2}\widetilde{w_N}(\tau)d\tau\\
&\quad +\eta(t)e^{(t-|t|)\xi^2}\int_{|\tau|\geq 1}\frac{e^{it\tau}}{i\tau+\xi^2}\widetilde{w_N}(\tau)d\tau - \eta(t)\int_{|\tau|\geq 1}\frac{e^{-|t|\xi^2}}{i\tau+\xi^2}\widetilde{w_N}(\tau)d\tau\\
&= I+II+III-IV
\end{align*}
where $w_N$ is defined by $\F_x(w_N)(\xi)=\varphi_N(\xi)\F_x(w)(\xi)$.

\vskip .3cm
\noindent
Contribution of $IV$. Clearly we have
$$\|P_L(IV)\|_{L^2_t}\lesssim \|P_L(\eta(t)e^{-|t|\xi^2})\|_{L^2_t}\int_{|\tau|\geq 1}\frac{|\widetilde{w_N}(\tau)|}{\cro{i\tau+\xi^2}}d\tau.$$
On the other hand, by Cauchy-Schwarz in $\tau $, 
$$\int_{|\tau|\geq 1}\frac{|\widetilde{w_N}(\tau)|}{\cro{i\tau+\xi^2}}d\tau \lesssim \sum_L\cro{L+N^2}^{-1}\|\varphi_L\widetilde{w_N}\|_{L^1_\tau}
\lesssim \sum_L \cro{L+N^2}^{-1/2}\|\varphi_L\widetilde{w_N}\|_{L^2_\tau},$$
which combined with (\ref{est-lin1}) yields the desired bound.

\vskip .3cm
\noindent
Contribution of II. By Cauchy-Schwarz inequality,
\begin{align}
\notag\|P_L(II)\|_{L^2_t} &\lesssim \|P_L(\eta(t)(e^{(t-|t|)\xi^2}-e^{-|t|\xi|^2}))\|_{L^2_t}\\
\notag &\quad\times \left(\int\frac{|\widetilde{w_N}(\tau)|^2}{\cro{i\tau+\xi^2}}d\tau\right)^{1/2}\left(\int_{|\tau|\leq 1}\frac{\cro{i\tau+\xi^2}}{|i\tau+\xi^2|^2}d\tau\right)^{1/2}\\
\notag &\lesssim \|P_L(\eta(t)(e^{(t-|t|)\xi^2}-e^{-|t|\xi|^2}))\|_{L^2_t}\\
\label{est-II}&\quad\times N^{-2}\cro{N}\sum_L\cro{L+N^2}^{-1/2}\|\varphi_L\widetilde{w_N}\|_{L^2_\tau}.
\end{align}
Hence we need to estimate
\begin{multline*}
\sum_L\cro{L+N^2}^{1/2}\|P_L(\eta(t)(e^{(t-|t|)\xi^2}-e^{-|t|\xi|^2}))\|_{L^2_t}\\ \lesssim \sum_L\cro{L+N^2}^{1/2}(\|P_L(\eta(t)e^{(t-|t|)\xi^2})\|_{L^2_t} +\|P_L(\eta(t)e^{-|t|\xi^2})\|_{L^2_t}).
\end{multline*}
The second term in the right-hand side is bounded by 1 thanks to estimate(\ref{est-lin1}). Denote $\theta(t)=\eta(t)e^{(t-|t|)\xi^2}$. It is not too hard to check that one integration by parts yields $|\hat{\theta}(\tau)|\lesssim \frac 1{|\tau|}$ whereas two integrations by parts give us $|\hat{\theta}(\tau)|\lesssim\frac{\cro{\xi}^2}{|\tau|^2}$.
We thus infer that
\begin{multline}
\sum_L\cro{L+N^2}^{1/2}\|\varphi_L\hat{\theta}\|_{L^2_\tau} \lesssim \sum_{L\leq 1}\cro{N}L^{1/2}\|\theta\|_{L^1_t}\\
\label{est-I}+\sum_{1\leq L\leq \cro{N}^2}\frac{\cro{N}}{L^{1/2}} +\sum_{L\geq \cro{N}^2}\cro{L}^{1/2}\frac{\cro{N}^2}{L^{3/2}}\lesssim \cro{N}.
\end{multline}
This provides the result for $N\geq 1$. In the case $N\leq 1$, we use a Taylor expansion and obtain
\begin{multline*}
\|P_L(\eta(t)(e^{(t-|t|)\xi^2}-1+1-e^{-|t|\xi^2})\|_{L^2_t}\\ \lesssim \sum_{n\geq 1}\frac{|\xi|^{2n}}{n!}\left(\|P_L(|t|^n\eta(t))\|_{L^2_t} + 2^n\|P_L(t^n\eta(t)\chi_{\R_-}(t))\|_{L^2_t}\right).
\end{multline*}
According to the Sobolev embedding $H^1\hookrightarrow B^{1/2}_{2,1}$ as well as the estimate $\|\chi_{\R_-}f\|_{H^1}\lesssim \|f\|_{H^1}$ provided $f(0)=0$, we deduce
\begin{align*}
&\sum_L\cro{L+N^2}^{1/2}\|P_L(\eta(t)(e^{(t-|t|)\xi^2}-e^{-|t|\xi^2}))\|_{L^2_t}\\
&\quad \lesssim \xi^2\sum_{n\geq 1}\frac 1{n!}(\||t|^n\eta(t)\|_{B^{1/2}_{2,1}}+2^n \|t^n\eta(t)\chi_{\R_-}(t)\|_{B^{1/2}_{2,1}})\\
&\quad \lesssim N^2\sum_{n\geq 1}\frac{ 2^n}{n!}\||t|^n\eta(t)\|_{H^1_t} \lesssim N^2.
\end{align*}
Gathering this and (\ref{est-II}) we conclude that 
$$\sum_L\cro{L+N^2}^{1/2}\|P_L(II)\|_{L^2_t}\lesssim \sum_L\cro{L+N^2}^{-1/2}\|\varphi_L\widetilde{w_N}\|_{L^2_\tau}.$$

\vskip .3cm
\noindent
Contribution of I. Since $I$ can be rewritten as
$$I = \eta(t)e^{(t-|t|)\xi^2}\int_{|\tau|\leq 1}\sum_{n\geq 1}\frac{(it\tau)^n}{n!}\frac{\widetilde{w_N}(\tau)}{i\tau+\xi^2}d\tau,$$
we have
$$\|P_L(I)\|_{L^2_t} \lesssim \sum_{n\geq 1}\frac 1{n!}\|P_L(t^n\theta(t))\|_{L^2_t}\int_{|\tau|\leq 1}\frac{|\tau|^n}{|i\tau+\xi^2|}|\widetilde{w_N}(\tau)|d\tau.$$
Using Cauchy-Schwarz we get, for $n\geq 1$,
\begin{align*}
\int_{|\tau|\leq 1}\frac{|\tau|^n}{|i\tau+\xi^2|}|\widetilde{w_N}(\tau)d\tau &\lesssim \left(\int\frac{|\widetilde{w_N}(\tau)|^2}{\cro{i\tau+\xi^2}}d\tau\right)^{1/2}
\left(\int_{|\tau|\leq 1}\frac{|\tau|^2\cro{i\tau+\xi^2}}{|i\tau+\xi^2|^2}d\tau\right)^{1/2}\\
&\lesssim \cro{N}^{-1}\sum_L\cro{L+N^2}^{-1/2}\|\varphi_L\widetilde{w_N}\|_{L^2_\tau}.
\end{align*}
Thus we see that it suffices to show that (see above  the contribution of II for the definition of $\theta$)
$$\sum_L\cro{L+N^2}^{1/2}\sum_{n\geq 1}\frac 1{n!}\|P_L(t^n\theta(t))\|_{L^2_t}\lesssim \cro{N}.$$
But again we have $|\F_t(t^n\theta(t))|\lesssim 2^n\min(\frac 1{|\tau|}, \frac{\cro{\xi}^2}{\tau^2})$ and arguing as in (\ref{est-I}), we get
$$\sum_L\cro{L+N^2}^{1/2}\sum_{n\geq 1}\frac 1{n!}\|P_L(t^n\theta(t))\|_{L^2_t}\lesssim \sum_{n\geq 1}\cro{N}\frac{2^n}{n!}\lesssim \cro{N}.$$

\vskip .3cm
\noindent
Contribution of III. Setting $\hat{g}(\tau) = \frac{\widetilde{w_N}(\tau)}{i\tau+\xi^2}\chi_{|\tau|\geq 1}$, we have to prove
\begin{equation}
\sum_L\cro{L+N^2}^{1/2}\|P_L(\theta g)\|_{L^2_t}\lesssim \sum_L\cro{L+N^2}^{1/2}\|P_Lg\|_{L^2_t}.
\end{equation}
Using the paraproduct decomposition, we have
$$P_L(\theta g) = P_L\Big(\sum_{M\gtrsim L}(P_{\lesssim M}\theta P_{\sim M}g + P_{\sim M}\theta P_{\lesssim M}g)\Big) = P_L(III_1)+P_L(III_2)$$
and we estimate the contributions of these two terms separately.

\vskip .3cm
\noindent
Contribution of $III_1$. The sum over $L\geq \cro{N}^2$ is estimated in the following way:
\begin{align*}
\sum_{L\geq \cro{N}^2}\cro{L+N^2}^{1/2}\|P_L(III_1)\|_{L^2_t} &\lesssim \sum_{L\geq \cro{N}^2}\cro{L}^{1/2}\sum_{M\gtrsim L}\|P_{\lesssim M}\theta\|_{L^\infty_t}\|P_M g\|_{L^2_t}\\
&\lesssim \sum_M\cro{M}^{1/2}\|P_Mg\|_{L^2_t}.
\end{align*}
Now we deal with the case where $L\lesssim \cro{N}^2$. If $\hat{\theta}$ is localized in an annulus $\{|\tau|\sim M\}$, we get from Bernstein inequality that
\begin{align}
\notag \sum_{L\leq \cro{N}^2}\cro{L+N^2}^{1/2}\sum_{M\gtrsim L}\|P_L(P_M\theta P_M g)\|_{L^2_t} &\lesssim \sum_M\cro{N}\sum_{L\lesssim M}L^{1/2}\|P_M\theta P_Mg\|_{L^1_t}\\
\notag &\lesssim \sum_M\cro{N}M^{1/2}\|P_M\theta\|_{L^2_t}\|P_Mg\|_{L^2}\\
\label{est-III1}&\lesssim \sum_M\cro{N}\|P_Mg\|_{L^2_t},
\end{align}
where we used the estimate $\|P_M\theta\|_{L^2_t}\lesssim \|\frac{\varphi_M(\tau)}{\tau}\|_{L^2_\tau}\lesssim M^{-1/2}$. If $\hat{\theta}$ is localized in a ball $\{|\tau|\ll M\}$, then we must have $M\sim L$ and thus
$$\sum_{L\leq \cro{N}^2}\cro{L+N^2}^{1/2}\sum_{M\sim L}\|P_L(P_{\ll M}\theta P_M g)\|_{L^2_t} \lesssim \sum_L\cro{N}\|P_{\ll L}\theta\|_{L^\infty_t}\|P_Lg\|_{L^2_t},$$
which is acceptable.

\vskip .3cm
\noindent
Contribution of $III_2$. Consider the case $L\geq \cro{N}^2$. Since $|\hat{\theta}|\lesssim\frac{\cro{\xi}^2}{\tau^2}$, we have
$$\|P_L(P_M\theta P_{\lesssim M}g)\|_{L^2_t}\lesssim \|\varphi_M\hat{\theta}\|_{L^1_\tau}\|P_{\lesssim M} g\|_{L^2_t} \lesssim \frac{\cro{N}^2}{M}\|g\|_{L^2_t}.$$
It follows that
$$\sum_{L\geq \cro{N^2}}\cro{L+N^2}^{1/2}\|P_L(III_2)\|_{L^2_t}\lesssim \sum_{M\gtrsim \cro{N}^2}M^{1/2}\frac{\cro{N}^2}M \|g\|_{L^2_t}\lesssim \cro{N}\|g\|_{L^2_t}.$$
It remains to establish the bound in the case $L\leq \cro{N}^2$. We may assume that $\hat{g}$ is supported in a ball $\{|\tau|\ll M\}$ since the other case has already been treated (cf. estimate (\ref{est-III1})).
Therefore, $M\sim L$ and
\begin{align*}
\sum_{L\leq \cro{N}^2}\cro{L+N^2}^{1/2}\sum_{M\gtrsim L}\|P_L(P_M\theta P_{\ll M}g)\|_{L^2_t} &\lesssim \sum_L\cro{N}\|P_L\theta P_{\ll L}g\|_{L^2_t}\\
&\lesssim \sum_L\cro{N}\|P_L\theta\|_{L^2_t}\sum_{M\ll L}\|P_Mg\|_{L^\infty_t}\\
&\lesssim \sum_L\cro{N} L^{-1/2}\sum_{M\ll L}M^{1/2}\|P_Mg\|_{L^2_t}\\
&\lesssim \sum_M\cro{N}\|P_Mg\|_{L^2_t}.
\end{align*}
The proof of Lemma \ref{lem-kxi} is complete.
\end{proof}

\begin{proposition}
\label{linearL}
Let $\L : f\to \L f$ denote the linear operator
\begin{eqnarray}
\L f(t,x) &=& \eta(t)\Bigl( \chi_{\R^+}(t)\int_0^tW(t-t',t-t')f(t')dt'\nonumber \\
 & & + \chi_{\R^-}(t)\int_0^tW(t-t',t+t')f(t')dt'\Bigr) \label{eq-L} \; .
\end{eqnarray}
 If $f\in \NN^{-1}$, then
\begin{equation}\label{est-linNhom}\| \L f \|_{\S^{-1}} \lesssim \|f\|_{\NN^{-1}}.\end{equation}
\end{proposition}

\begin{proof}

It suffices to show that
\begin{equation}\label{est-lin2}\|\L f\|_{X^{-1,\frac 12,1}}\lesssim \|f\|_{X^{-1,-\frac 12,-1}}\end{equation}
and
\begin{equation}\label{est-lin3}\|\L f\|_{Y^{-1,\frac 12}}\lesssim \|f\|_{Y^{-1,-\frac 12}}.\end{equation}
Taking the $ x$-Fourier transform,  we get
\arraycolsep1pt
\begin{align*}
\L f(t,x)  & =  U(t)
\Biggl[\chi_{\R+}(t) \,   \eta(t) \,\int_{\R} e^{i x \xi} \int_0^t
 e^{-|t-t'| \xi^2} {\cal F}_x (U(-t') f(t') )(\xi) \, dt' d\xi \\
& \quad +
\chi_{\R-}(t) \,   \eta(t) \,\int_{\R}  e^{i x \xi} \int_0^t
 e^{-|t+t'| \xi^2} {\cal F}_x (U(-t') f(t') )(\xi) \, dt' d\xi \Biggr] \\
 & =  U(t)
\Biggl[  \eta(t) \,\int_{\R}  e^{i x \xi} \int_0^t
 e^{-|t| \xi^2}  e^{t' \xi^2}{\cal F}_x (U(-t') f(t') )(\xi) \, dt' d\xi \Biggr] \; .
 \end{align*}
 Setting $w(t') = U(-t')f(t')$, and using the time Fourier transform, we infer that
 $$\L f(t,x) = U(t)\left[\eta(t)\int_{\R^2}e^{ix\xi}\frac{e^{it\tau}e^{(t-|t|)\xi^2}-e^{-|t|\xi^2}}{i\tau+\xi^2}\tilde{w}(\tau,\xi)d\tau d\xi\right].$$
 Estimate (\ref{est-lin2}) follows then easily from Lemma \ref{lem-kxi}.

Now we turn to estimate (\ref{est-lin3}). After square summing, it suffices to prove that for any dyadic $N$,
\begin{equation}\label{est-linY}\|(\partial_t+\partial_{xxx}-\partial_{xx}+I)P_N\L f\|_{L^1_tL^2_x}\lesssim \|P_Nf\|_{L^1_tL^2_x}.\end{equation}
In view of the expression of $ \L $ it suffices to prove  \re{est-linY} separately for $ \chi_{\R^+} \L f$ and $\chi_{\R^-} \L f$
First, a straightforward calculation leads to
\begin{multline*}(\partial_t+\partial_{xxx}-\partial_{xx}+I)(\chi_{\R^+} \L f(t) )\\= \eta(t)\chi_{\R^+}(t)f(t)+(\eta'(t)+\eta(t))\chi_{\R^+}(t)\int_0^tW(t-t',t-t')f(t')dt'.\end{multline*}
Computing the $L^1_tL^2_x$ norm, we get
\begin{multline*}\|(\partial_t+\partial_{xxx}-\partial_{xx}+I)P_N (\chi_{\R^+} \L f)\|_{L^1_tL^2_x} \\\lesssim \|f\|_{L^1_tL^2_x}+ \|\eta'+\eta\|_{L^1_t}\sup_t\int_0^\infty\|e^{i(t-t')\xi^3}e^{-(t-t')\xi^2}\widehat{f}(t')\|_{L^2_\xi}dt',\end{multline*}
and estimate (\ref{est-linY}) follows.

Now, let us tackle the proof for $\chi_{\R^-} \L f$. We have to work a little more since clearly $ \L f $ does not satisfy the same equation for negative times.
 Actually, one can check that
 \begin{multline}(\partial_t+\partial_{xxx}+\partial_{xx}+I)(\chi_{\R^-} \L f(t) )\\
 =
 \eta(t)\chi_{\R^-}(t) W(0,2t) f(t)+(\eta'(t)+\eta(t))\chi_{\R^-}(t)\int_0^tW(t-t',t+t')f(t')dt'.
 \label{nh}\end{multline}
 and thus
  \begin{multline}(\partial_t+\partial_{xxx}-\partial_{xx}+I)(\chi_{\R^-} \L f(t) )
 = -2\partial_{xx} (\chi_{\R^-} \L f(t) )\\
+ \eta(t)\chi_{\R^-}(t)W(2t,0) f(t)+(\eta'(t)+\eta(t))\chi_{\R^-}(t)\int_0^tW(t-t',t+t')f(t')dt'.
\label{eses}\end{multline}
Setting $ w:=P_N( \chi_{\R^-} \L f(t)) $ and
$$
g:= \eta(t)\chi_{\R^-}(t)W(2t,0) f(t)+(\eta'(t)+\eta(t))\chi_{\R^-}(t)\int_0^tW(t-t',t+t')f(t')dt'
$$
we first note as above that
 \begin{equation}
 \| g\|_{L^1_t L^2_x} \lesssim \|f\|_{L^1_tL^2_x}\; . \label{zszs}
 \end{equation}
 Now, according to \re{nh}, $ w$ satisfies
 $$
 w_t -w_{xxx}+w_{xx}+w= g
 $$
 Taking the $ L^2_x $-scalar product with $ w$ and using Cauchy-Schwarz yield
 \begin{equation}
 \frac{1}{2}\frac{d}{dt} \|w\|^2_{L^2_x}- \| w_x\|_{L^2_x}^2+\|w\|^2_{L^2_x}\ge -\|g\|_{L^2_x} \|w\|_{L^2_x}\; .\label{nh1}
 \end{equation}
 By the  frequencies localization of $w$ and Bernstein inequality,$ \|w_x\|_{L^2_x}\ge \frac{1}{2} N \|w\|_{L^2_x} $.
 Therefore,  for $ t>0 ,$ such that $ \|w(t)\|_{L^2_x} \neq 0 $, we can divide \re{nh1} by $ \|w(t)\|_{L^2_x}  $  to get 
 \begin{equation} \label{ds}
 N^2 \|w(t)\|_{L^2_x}\lesssim \frac{d}{dt} \|w(t)\|_{L^2_x}+\|w(t)\|_{L^2_x}+\|g(t)\|_{L^2_x} 
 \end{equation}
 On the other hand, for $ t>0 $, the smoothness and non negativity  of $t\mapsto \|w(t)\|_{L^2_x} $ forces $ \frac{d}{dt} \|w(t)\|^2_{L^2_x}=0 $ as soon as $ \|w(t)\|_{L^2_x} = 0 $. This ensures that \re{ds} is actually valid for all $ t>0 $. 
 Therefore  integrating \re{ds} on $]0,t[ $  we infer that
 $$
  \|w_{xx} \|_{L^1_t L^2_x}  \sim N^2 \|w\|_{L^1_t L^2_x}\lesssim \|w\|_{L^\infty_t L^2_x} + \|w\|_{L^1_tL^2_x} +  \| g\|_{L^1_t L^2_x}\; .
 $$
Since obviously,
$$
\|w\|_{L^1_tL^2_x} + \|w\|_{L^\infty_t L^2_x} \lesssim \sup_t\int_0^\infty\|e^{i(t-t')\xi^3}e^{-|t+t'|\xi^2}\widehat{P_N f}(t')\|_{L^2_\xi}dt'\lesssim \|P_N f\|_{L^1_t L^2_x}
$$
it follows that
$$
\|w_{xx} \|_{L^1_t L^2_x} \lesssim \|P_N f\|_{L^1_t L^2_x}
$$
which concludes the proof together with (\ref{eses}) and (\ref{zszs}).
\end{proof}

\section{Bilinear estimate}\label{sectionbilinear}
In this section we provide a proof of the following crucial bilinear estimate.
\begin{proposition}\label{propo}
For all $u,v\in\S^{-1}$, we have
\begin{equation}\label{est-bil}\|\partial_x(uv)\|_{\NN^{-1}}\lesssim \|u\|_{\S^{-1}}\|v\|_{\S^{-1}}.\end{equation}
\end{proposition}

First we remark that because of the $L^2_\xi$ structure of the spaces involved in our analysis we have the following localization property
$$\|f\|_{\S^{-1}}\sim \Big(\sum_N\|P_Nf\|_{\S^{-1}}^2\Big)^{1/2}\quad\textrm{and}\quad \|f\|_{\NN^{-1}}\sim \Big(\sum_N\|P_Nf\|_{\NN^{-1}}^2\Big)^{1/2}.$$
Performing a dyadic decomposition for $u, v$ we thus obtain
\begin{equation}\label{eq-dec}\|\partial_x(uv)\|_{\NN^{-1}}\sim \Big(\sum_N\Big\|\sum_{N_1,N_2}P_N\partial_x(P_{N_1}uP_{N_2}v)\Big\|_{\NN^{-1}}^2\Big)^{1/2}.\end{equation}
We can now reduce the number of case to analyze by noting that the right-hand side vanishes unless one of the following cases holds:
\begin{list}{$\bullet$}
\item (high-low interaction) $N\sim N_2$ and $N_1\lesssim N$,
\item \item (low-high interaction) $N\sim N_1$ and $N_2\lesssim N$,
\item (high-high interaction) $N\ll N_1\sim N_2$.
\end{list}
\noindent
The former two cases are symmetric. In the first case, we can rewrite the right-hand side of (\ref{eq-dec}) as
$$\|\partial_x(uv)\|_{\NN^{-1}}\sim \Big(\sum_N\|P_N\partial_x(P_{\lesssim N}u P_Nv)\|_{\NN^{-1}}^2\Big)^{1/2},$$
and it suffices to prove the high-low estimate
\begin{equation}\label{HL}\tag{HL}\|P_N\partial_x(P_{\lesssim N}uP_Nv)\|_{\NN^{-1}}\lesssim \|u\|_{\S^{-1}}\|P_Nv\|_{\S^{-1}}\end{equation}
for any dyadic $N$. If we consider now the third case, we easily get
$$\|\partial_x(uv)\|_{\NN^{-1}}\lesssim \sum_{N_1}\|P_{\ll N_1}\partial_x(P_{N_1}uP_{N_1}v)\|_{\NN^{-1}},$$
and it suffices to prove for any $N_1$ the high-high estimate
\begin{equation}\label{HH}\tag{HH}\|P_{\ll N_1}\partial_x(P_{N_1}uP_{N_1}v)\|_{\NN^{-1}}\lesssim \|P_{N_1}u\|_{\S^{-1}}\|P_{N_1}v\|_{\S^{-1}}\end{equation}
since the claim follows then from Cauchy-Schwarz.

\subsection{Proof of (\ref{HL})}
We decompose the bilinear term as $$P_N\partial_x(P_{\lesssim N}uP_Nv) = \sum_{N_1\lesssim N}\sum_{L, L_1, L_2}P_NQ_L\partial_x(P_{N_1}Q_{L_1}uP_NQ_{L_2}v).$$
Using the well-known resonance relation
\begin{equation}\label{eq-smooth}\xi_1^3+\xi_2^3+\xi_3^3=3\xi_1\xi_2\xi_3\quad\textrm{whenever}\quad \xi_1+\xi_2+\xi_3=0,\end{equation}
we see that non-trivial interactions only happen when
\begin{equation}\label{eq-LmaxHL}L_{max}\sim \max(N^2N_1,L_{med})\end{equation}
where $L_{max}\geq L_{med}\geq L_{min}$ holds for $L, L_1, L_2$.

First we consider the easiest case $N_1\lesssim 1$. We take advantage of the $Y^{-1,-\frac 12}$ part of $\NN^{-1}$ as well as H\"{o}lder and Bernstein inequalities to obtain
\begin{align*}
\sum_{N_1\lesssim 1}\|P_N\partial_x(P_{N_1}uP_Nv)\|_{Y^{-1,-\frac 12}} &\lesssim \sum_{N_1\lesssim 1}\cro{N}^{-1}N\|P_N(P_{N_1}uP_Nv)\|_{L^1_tL^2_x}\\
&\lesssim \sum_{N_1\lesssim 1}\|P_{N_1}u\|_{L^2_tL^\infty_x}\|P_Nv\|_{L^2}\\
&\lesssim \sum_{N_1\lesssim 1}N_1^{1/2}\|P_{N_1}u\|_{L^2}\|P_Nv\|_{L^2}\\
&\lesssim \|u\|_{\S^{-1}}\|v\|_{\S^{-1}}
\end{align*}
where we used (\ref{est-L2S-1}) in the last estimate.
One can now assume we have large space frequencies, i.e. $N\gtrsim N_1\gtrsim 1$.

\subsubsection{Case $L_{max}=L$}\label{sec-L-HL}
In light of (\ref{eq-LmaxHL}), we are in the region $L\gtrsim N^2N_1$. From the definition of $X^{-1,\frac 12,1}$ we have
\begin{align*}
&\sum_{1\lesssim N_1\lesssim N}\sum_{L\gtrsim N^2N_1}\|P_NQ_L\partial_x(P_{N_1}uP_Nv)\|_{X^{-1,-\frac 12,1}}\\
&\quad \lesssim \sum_{1\lesssim N_1\lesssim N}\sum_{L\gtrsim N^2N_1}N^{-1}\cro{L}^{-1/2}N\|P_NQ_L(P_{N_1}uP_Nv)\|_{L^2}.
\end{align*}
Then, estimates (\ref{est-L2S-1}) and (\ref{est-Lit}) lead to the bound
\begin{align*}
& \lesssim \sum_{1\lesssim N_1\lesssim N}N^{-1}N_1^{-1/2}\|P_{N_1}u\|_{L^\infty_tL^2_x}\|P_Nv\|_{L^2_tL^\infty_x}\\
& \lesssim \sum_{1\lesssim N_1\lesssim N}N_1^{1/2}N^{-1/2} \|P_{N_1}u\|_{L^\infty_tH^{-1}_x} \|P_Nv\|_{L^2}\\
&\lesssim \|u\|_{\S^{-1}}\|P_Nv\|_{\S^{-1}}.
\end{align*}

\subsubsection{Case $L_{max}=L_1$}\label{sec-LmaxL1}
Here we must have either $L_1\sim N^2N_1$ or $L_1\sim L_{med}$. Note that the second case has been treated in Subsection \ref{sec-L-HL} when $L_{med}=L$ and we reduce to $L_1\sim L_2$.
The contribution for the former case can be estimated as follows:
\begin{align*}
&\sum_{1\lesssim N_1\lesssim N}\sum_{L_1\sim N^2N_1}\|P_NQ_L\partial_x(P_{N_1}Q_{L_1}uP_Nv)\|_{Y^{-1,-\frac 12}}\\
&\quad\lesssim \sum_{1\lesssim N_1\lesssim N}\|P_{N_1}Q_{N^2N_1}uP_Nv\|_{L^1_tL^2_x}\\
&\quad\lesssim \sum_{1\lesssim N_1\lesssim N}N_1^{1/2}\|P_{N_1}Q_{N^2N_1}u\|_{L^2}\|P_Nv\|_{L^2}.
\end{align*}
Now we can exploit the smoothing relation $L_1\sim N^2N_1$ and obtain
\begin{equation}\label{eq-srHL}N_1^{1/2}\|P_{N_1}Q_{N^2N_1}u\|_{L^2} \lesssim N^{-1}N_1(N_1^{-1}L_1^{1/2}\|P_{N_1}Q_{L_1}u\|_{L^2}),\end{equation}
which combined with (\ref{est-L2S-1}), (\ref{est-L2l2}) and Cauchy-Schwarz in $N_1$ yields the desired bound.

It remains to treat the case $L_1\sim L_2\gtrsim N^2N_1$ where we can use both on $L_1$ and $L_2$ the smoothing relation. Arguing as before we get
\begin{align*}
&\sum_{1\lesssim N_1\lesssim N}\sum_{L_1\sim L_2\gtrsim N^2N_1}\|P_N\partial_x(P_{N_1}Q_{L_1}uP_NQ_{L_2}v)\|_{Y^{-1,-\frac 12}}\\
&\quad\lesssim \sum_{1\lesssim N_1\lesssim N}\sum_{L_1\gtrsim N^2N_1}N_1^{1/2}\|P_{N_1}Q_{L_1}u\|_{L^2}\|P_NQ_{L_1}v\|_{L^2}.
\end{align*}
In this regime, (\ref{eq-srHL}) is still valid if we replace $Q_{N^2N_1}$ by $Q_{L_1}$. Applied on $u$ and $v$, this provides the bound
\begin{align*}
&\quad\lesssim \sum_{1\lesssim N_1\lesssim N}N_1^{1/2}N^{-1}\Big(\sum_{L_1}(N_1^{-1}L_1^{1/2}\|P_{N_1}Q_{L_1}u\|_{L^2})^2\Big)^{1/2}
\Big(\sum_{L_1}(N^{-1}L_1^{1/2}\|P_NQ_{L_1}v\|_{L^2})^2\Big)^{1/2}\\
&\quad\lesssim \sum_{1\lesssim N_1\lesssim N}N_1^{1/2}N^{-1}\|P_{N_1}u\|_{\S^{-1}}\|P_Nv\|_{\S^{-1}},
\end{align*}
which is acceptable (with about $N^{-1/2}$ of spare).

\subsubsection{Case $L_{max}=L_2$}
By (\ref{eq-LmaxHL}), it suffices to consider the case $L_2\sim N^2N_1$. With a similar argument we get
\begin{align*}
&\sum_{N_1\lesssim N}\sum_{L_2\sim N^2N_1}\|P_N\partial_x(P_{N_1}uP_NQ_{L_2}v)\|_{Y^{-1,-\frac 12}}\\
&\quad\lesssim \sum_{N_1\lesssim N}N_1^{1/2}\|P_{N_1}u\|_{L^2}\|P_NQ_{N^2N_1}v\|_{L^2}\\
&\quad\lesssim \Big(\sum_{N_1}\|P_{N_1}u\|_{L^2}^2\Big)^{1/2}\Big(\sum_{N_1}(N_1^{1/2}\|P_NQ_{N^2N_1}v\|_{L^2})^2\Big)^{1/2}\\
&\quad\lesssim \|u\|_{\S^{-1}}\Big(\sum_{L_2}(N^{-1}L_2^{1/2}\|P_NQ_{L_2}v\|_{L^2})^2\Big)^{1/2},
\end{align*}
which achieves the proof of (\ref{HL}).

\subsection{Proof of (\ref{HH})}
Performing the decomposition
$$P_{\ll N_1}\partial_x(P_{N_1}uP_{N_1}v) = \sum_{N\ll N_1}\sum_{L,L_1,L_2}P_NQ_L\partial_x(P_{N_1}Q_{L_1}uP_{N_1}Q_{L_2}v),$$
we see from (\ref{eq-smooth}) that we may restrict ourself to the region where \begin{equation}\label{eq-LmaxHH}L_{max}\sim \max(N_1^2N,L_{med}).\end{equation}
Moreover, we may assume by symmetry that $L_1\geq L_2$.
Low frequencies $N\lesssim 1$ are easily handled:
\begin{align*}
\sum_{N\lesssim 1}\|P_N\partial_x(P_{N_1}uP_{N_1}v)\|_{Y^{-1,-\frac 12}} &\lesssim \sum_{N\lesssim 1}\cro{N}^{-1}N\|P_N(P_{N_1}uP_{N_1}v)\|_{L^1_tL^2_x}\\
&\lesssim \sum_{N\lesssim 1}\|P_{N_1}u\|_{L^2}\|P_{N_1}v\|_{L^2}\\
&\lesssim \|P_{N_1}u\|_{\S^{-1}}\|P_{N_1}v\|_{\S^{-1}}.
\end{align*}
Therefore it is sufficient to consider $N_1\gg N\gtrsim 1$.
\subsubsection{Case $L_{max}=L$}
In this region one has $L\gtrsim N_1^2N$. Let us assume $L_1\lesssim N_1^2N^{1-\eps}$ for some $\eps>0$ so that we wish to bound
\begin{equation}\label{temr-HH1}
\Big\|\sum_{1\lesssim N\ll N_1}P_NQ_{\gtrsim N_1^2N}\partial_x(P_{N_1}Q_{\lesssim N_1^2N^{1-\eps}}uP_{N_1}v)\Big\|_{X^{-1,-\frac 12,1}}.\end{equation}
Using the triangle inequality we reduce to estimate
$$
\sum_{1\lesssim N\ll N_1}\sum_{\substack{L\gtrsim N_1^2N\\ L_1\lesssim N_1^2N^{1-\eps}}}L^{-1/2}\|P_{N_1}Q_{L_1}uP_{N_1}v\|_{L^2}.$$
In order to get a suitable control for this term, we apply the Kato smoothing effect (\ref{est-smooth}) together with estimate (\ref{est-L2S-1}) to get
\begin{align*}
\|P_{N_1}Q_{L_1}uP_{N_1}v\|_{L^2}& \lesssim \|P_{N_1}Q_{L_1}u\|_{L^2_xL^\infty_t}
\|P_{N_1}v\|_{L^\infty_xL^2_t}\\ &\lesssim L_1^{1/2} \|P_{N_1}u\|_{\S^{-1}}\|P_{N_1}v\|_{\S^{-1}}.
\end{align*}
Therefore it remains to establish
$$\sum_{1\lesssim N\ll N_1}\sum_{\substack{L\gtrsim N_1^2N\\ L_1\lesssim N_1^2N^{1-\eps}}}
L^{-1/2}L_1^{1/2}\lesssim 1,$$
but this is easily verified by Schur's test for any $\eps>0$. The situation where $L_2\lesssim N_1^2N^{1-\eps}$
is identical to the previous case and we suppose now $L_1,L_2\gtrsim N_1^2N^{1-\eps}$. Estimating the $\NN^{-1}$-norm by the $Y^{-1,-\frac 12}$-norm, and using the H\"{o}lder and Bernstein inequalities we see that the contribution in this case is bounded by

\begin{align}\notag &  \sum_{1\lesssim N\ll N_1}\|P_N(P_{N_1}Q_{\gtrsim N_1^2N^{1-\eps}}uP_{N_1}Q_{\gtrsim N_1^2N^{1-\eps}}v\|_{L^1_tL^2_x}\\
\label{est-HH2} & \quad\lesssim \sum_{1\lesssim N\ll N_1}N^{1/2}\|P_{N_1}Q_{\gtrsim N_1^2N^{1-\eps}}u\|_{L^2}\|P_{N_1}Q_{\gtrsim N_1^2N^{1-\eps}}v\|_{L^2}.
\end{align}
On the other hand the resonance relation and (\ref{est-L2l2}) yield
\begin{align*}N^{1/2}\|P_{N_1}Q_{\gtrsim N_1^2N^{1-\eps}}u\|_{L^2} &\lesssim N^{\eps/2} \Big(\sum_{L_1}[N_1^{-1}L_1^{1/2}\|P_{N_1}Q_{L_1}u\|_{L^2}]^2\Big)^{1/2}\\&\lesssim N^{\eps/2}\|P_{N_1}u\|_{\S^{-1}},\end{align*}
and similarly for $v$. Inserting this into (\ref{est-HH2}) we deduce
$$
(\ref{est-HH2}) \lesssim \sum_{N\gtrsim 1}N^{-1/2+\eps}\|P_{N_1}u\|_{\S^{-1}}\|P_{N_1}v\|_{\S^{-1}},
$$
which is acceptable for $\eps<1/2$.

\subsubsection{Case $L_{max}=L_1$}
First we consider the region $L_1\sim N_1^2N$ and we want to estimate
\begin{align*}
&\|\sum_{N\ll N_1}P_N\partial_x(P_{N_1}Q_{N_1^2N}uP_{N_1}v)\|_{Y^{-1,-\frac 12}}\\
& \quad\lesssim \Big(\sum_{N}[N^{1/2}\|P_{N_1}Q_{N_1^2N}u\|_{L^2}\|P_{N_1}v\|_{L^2}]^2\Big)^{1/2}
\end{align*}
where we took care of not using the triangle inequality in order to keep the $\ell^2$-norm in $N$.  The term $\|P_{N_1}u\|_{L^2}$ can be handled with help of (\ref{est-L2S-1}), while the change of variable $N\sim L_1N_1^{-2}$ for fixed $N_1$ leads to the bound
$$
\lesssim \Big(\sum_{L_1} [N_1^{-1}L_1^{1/2}\|P_{N_1}Q_{L_1}u\|_{L^2}]^2\Big)^{1/2}\|P_{N_1}v\|_{\S^{-1}}
\lesssim \|P_{N_1}u\|_{\S^{-1}}\|P_{N_1}v\|_{\S^{-1}}.
$$
Finally in the case $L_1\sim L_2\gtrsim N_1^2N$, arguing as in Subsection \ref{sec-LmaxL1}, we get
\begin{align*}
&\Big\|\sum_{1\lesssim N\ll N_1}\sum_{L_1\sim L_2\gg N_1^2N}P_N\partial_x(P_{N_1}Q_{L_1}uP_{N_1}Q_{L_2}v)\Big\|_{Y^{-1,-\frac 12}}\\
& \quad\lesssim \sum_{1\lesssim N\ll N_1}\sum_{L_1\gg N_1^2N}N^{1/2}\|P_{N_1}Q_{L_1}u\|_{L^2}\|P_{N_1}Q_{L_1}v\|_{L^2}\\
& \quad\lesssim \sum_{N\gtrsim 1} N^{-1/2}\Big(\sum_{L_1}(N_1^{-1}L_1^{1/2}\|P_{N_1}Q_{L_1}u\|_{L^2})^2\Big)^{1/2}\Big(\sum_{L_1}(N_1^{-1}L_1^{1/2}\|P_{N_1}Q_{L_1}v\|_{L^2})^2\Big)^{1/2},
\end{align*}
which is acceptable (with about $N^{-1/2}$ of spare).
\section{Well-posedness}\label{sectionwell}
In this section, we prove the well-posedness result. Using a standard fixed point procedure, it is clear that the bilinear estimate (\ref{est-bil}) allows us to show local well-posedness but for small initial data only. This is because $H^{-1}$ appears as a critical space for KdV-Burgers and thus we can't get the desired contraction factor in our estimates. In order to remove the size restriction on the data, we need to change the metric on our resolution space.

For $\beta\ge 1 $, let us define the following norm on $ \S^{-1} $,
$$\|u\|_{\ZZ_\beta} = \inf_{\substack{u=u_1+u_2\\ u_1\in\S^{-1}, u_2\in \S^0}}\left\{\|u_1\|_{\S^{-1}}+\frac 1\beta \|u_2\|_{\S^0}\right\}.$$
Note that this norm is equivalent to $ \| \cdot \|_{\S^{-1}} $.   Now we will need the following modification
 of Proposition \ref{propo}. This new proposition means that as soon as we assume more regularity  on $ u$ we can get
  a contractive factor for small times in the bilinear estimate.
 \begin{proposition}
There exists $ \nu>0 $ such that for all $(u,v)\in\S^{0}\times \S^{-1} $, with compact support (in time) in $[-T,T]$,
 it holds\
\begin{equation}\label{est-bil3}\|\partial_x(uv)\|_{\NN^{-1}}\lesssim T^\nu \|u\|_{\S^{0}}\|v\|_{\S^{-1}}.
\end{equation}
 \end{proposition}
 \begin{proof}
It suffices to slightly modify  the proof of Proposition \ref{propo} to make use of  the following result that can be found in   [\cite{GTV}, Lemma 3.1] (see
 also  [\cite{MR2}, Lemma 3.6]):
 For any $\theta >0$, there exists $\mu=\mu(\theta)>0$ such that for any smooth function  $ f$ with compact support in
  time in $[-T,T]$,
\begin{equation}
 \left\| {\mathcal F}^{-1}_{t,x}\left( \frac{\hat{f}(\tau,\xi)} {\langle\tau - \xi^3\rangle^{\theta}}\right) \right\|_{L^{2}_{t,x}}
 \lesssim
   T^\mu \|f\|_{L^{2,2}_{t,x}} \;.
 \label{strichartz}
 \end{equation}
 According to \re{est-L2l2} this  ensures, in particular,  that  for any $ w\in \S^0 $ with compact support in $[-T,T] $  it holds
  \begin{equation}\label{lala}
 \| w\|_{L^2_t H^{3/4}}\lesssim  \| w\|_{X^{0,3/8,2}}\lesssim
 T^{\mu(\frac{1}{8})}  \| w\|_{X^{0,1/2,2}}\lesssim  T^{\mu(\frac{1}{8})} \|w\|_{\S^0} \; .
 \end{equation}
  It is pretty clear that the interactions between high  frequencies of $ u $ and high or low frequencies of $ v$
   can be treated by following   the proof of Proposition \ref{propo}
   and using \re{lala}. The  region that seems the most dangerous is   the one of interactions between low frequencies of
    $ u $ and  high frequencies of $ v$, that is the region of $ (HL)$ in the proof of Proposition \ref{propo}.
     But actually this region can also be easily treated. For instance in the case \ref{sec-L-HL} it suffices to notice that
     \arraycolsep2pt
      \begin{eqnarray*}
\sum_{1\lesssim N_1\lesssim N} & &\sum_{L\gtrsim N^2N_1} \|P_NQ_L\partial_x(P_{N_1}u P_Nv)\|_{X^{-1,-\frac 12,1}}\\
&&\ \lesssim \sum_{1\lesssim N_1\lesssim N}\sum_{L\gtrsim N^2N_1}  N^{-1}\cro{L}^{-1/2}
N\|P_NQ_L(P_{N_1} u P_Nv)\|_{L^2}.\\
&& \lesssim  \sum_{1\lesssim N_1\lesssim N}
N^{-1}N_1^{-1/2} \| P_{N_1} u\|_{L^2_t L^\infty_x }   \|P_N v\|_{L^\infty_t L^2_x}\\
&& \lesssim \sum_{1\lesssim N_1\lesssim N} N_1^{-1/2}
 \| P_{N_1} u\|_{L^2_t H^{1/2}_x }  \|P_N v\|_{L^\infty_t H^{-1}_x}\\
 && \lesssim T^{\mu(\frac{1}{8})}   \|u\|_{\S^0} \|v\|_{\S^{-1}}
\end{eqnarray*}
  and in the case \ref{sec-LmaxL1} it suffices to replace \re{eq-srHL} by simply
  $$
 N_1^{1/2}\|P_{N_1}Q_{N^2N_1}u\|_{L^2} \lesssim N_1^{-1/4} \| P_{N_1} u\|_{L^2_t H^{3/4}_x } \lesssim
  N_1^{-1/4} T^{\mu(\frac{1}{8})}   \|u\|_{\S^0}\; .
  $$
  The other cases can be handle in a similar way.
  \end{proof}
We are now in position   to prove that the application
$$F_\phi^T : u\mapsto \eta(t)\Big[W(t)\phi-\frac 12\L\partial_x(\eta_T u)^2\Big],$$
where $\L$ is defined in (\ref{eq-L}), is contractive on a ball of $\ZZ_\beta$ for a suitable $\beta>0 $
 and $ T>0 $ small enough. Assuming this for a while, the  local part of Theorem \ref{wellposed}  follows by using standard arguments. Note that the uniqueness will hold in the restriction spaces 
   ${\mathcal S}^{-1}_\tau $ endowed with the norm 
 $$
 \|u\|_{ {\mathcal S}^{-1}_\tau}:=\inf_ {v\in {\mathcal S}^{-1}} \{ \|v\|_{{{\mathcal S}^{-1}}}, \; v\equiv u \mbox{ on } [0,\tau] \}\; .
 $$
 Finally, to see that the solution $ u $ can be extended for all positive times and belongs to $C(\R_+^*;H^\infty) $ it suffices to notice that, according to 
  \re{est-L2S-1}, 
  $u\in {\mathcal S}^{-1}_\tau \hookrightarrow L^2(]0,\tau[\times \R)$ . Therefore, for any $0<\tau'<\tau $  there exists  
  $t_0\in ]0,\tau'[ $, such that $ u(t_0)$ belongs  to $L^2(\R)$ . Since according to   \cite{MR2}, \re{KdVB} is globally  well-posed in $ L^2(\R) $ with a solution belonging to $ C(\R^*_+;H^\infty(\R)) $,  the conclusion follows.

  In order to prove that $ F_\phi^T $ is contractive, the first step is  to establish the following result.
\begin{proposition}\label{prop-bilZ} For any  $ \beta\ge 1 $ there   exists $0<T=T(\beta)<1$ such that for any $u,v\in\ZZ_{\beta}$ with compact support in $[-T,T]$ we have
\begin{equation}\label{est-bilZ}
\|\L\partial_x(uv)\|_{\ZZ_\beta}\lesssim \|u\|_{\ZZ_\beta}\|v\|_{\ZZ_\beta}.
\end{equation}
\end{proposition}
Assume for the moment that (\ref{est-bilZ}) holds and let $u_0\in H^{-1}$ and $\alpha>0$.
Split the data $u_0$ into low and high frequencies:
$$u_0=P_{\lesssim N}u_0+P_{\gg N}u_0$$
for a dyadic number $N$. Taking $N=N(\alpha)$ large enough, it is obvious to check that $\|P_{\gg N}u_0\|_{H^{-1}}\leq \alpha$. Hence, according to \re{est-lin},
$$\|\eta(\cdot)W(\cdot)P_{\gg N}u_0\|_{\ZZ_\beta}\lesssim \alpha.$$
Using now the $\S^0 $-part of $\ZZ_\beta$, we control the low frequencies as follows:
$$\|\eta(\cdot)W(\cdot)P_{\lesssim N}u_0\|_{\S^0} \lesssim \frac 1\beta \|P_{\lesssim N}u_0\|_{L^2}\lesssim \frac{N}{\beta}\|u_0\|_{H^{-1}}.$$
Thus we get
$$\|\eta(\cdot)W(\cdot)P_{\lesssim N}u_0\|_{\ZZ_\beta}\lesssim \alpha\ \textrm{ for }\ \beta\gtrsim\frac{N \|u_0\|_{H^{-1}}}{\alpha}.$$
Since $\alpha$ can be chosen as small as needed, we conclude with (\ref{est-bilZ}) that $F_\phi^T$ is contractive on a ball of $\ZZ_\beta$ of radius $R\sim \alpha$ as soon as $ \beta \gtrsim  N \|u_0\|_{H^{-1}}/\alpha $
 and $T= T(\beta)$.

\begin{proof}[Proof of Proposition \ref{prop-bilZ}] By definition on the function space $ {\mathcal Z}_\beta $,
there exist $u_1,v_1\in\S^{-1}$ and $u_2,v_2\in \S^0 $ such that $u=u_1+v_1$, $v=v_1+v_2$ and
\begin{align*}\|u_1\|_{\S^{-1}}+\frac 1\beta\|u_2\|_{\S^0}\leq & 2\|u\|_{\ZZ_\beta},\\ \|v_1\|_{\S^{-1}}+\frac 1\beta\|v_2\|_{\S^0}\leq & 2\|v\|_{\ZZ_\beta}.
\end{align*}
Thus one can decompose the left-hand side of (\ref{est-bilZ}) as
\begin{align*}
\|\L\partial_x(uv)\|_{\ZZ_\beta} &\lesssim \|\L\partial_x(u_1v_1)\|_{\S^{-1}}+\|\L\partial_x(u_1v_2+u_2v_1)\|_{\S^{-1}}+\|\L\partial_x(u_2v_2)\|_{\S^{-1}}\\
&= I+II+III.
\end{align*}
From the estimates (\ref{est-linNhom}) and (\ref{est-bil}) we get
$$I\lesssim  \|\partial_x(u_1v_1)\|_{\NN^{-1}}\lesssim \|u_1\|_{S^{-1}}\|v_1\|_{\S^{-1}}\lesssim \|u\|_{\ZZ_\beta}\|v\|_{\ZZ_\beta}.$$
On the other hand, we obtain from (\ref{est-bil3}) that
$$III\lesssim T^{\nu}\|u_2\|_{\S^0}\|v_2\|_{\S^0}\lesssim \beta^2 T^{\nu}\|u\|_{\ZZ_\beta}
\|v\|_{\ZZ_\beta}.$$
and
\begin{eqnarray*}
II &  \lesssim &  T^{\nu}(\|u_1\|_{\S^{-1}}\|v_2\|_{\S^0}+\|u_2\|_{\S^0}\|v_1\|_{\S^{-1}})\\
&\lesssim &  \beta T^{\nu}\|u\|_{\ZZ_\beta}\|v\|_{\ZZ_\beta}.
\end{eqnarray*}
We thus get
$$\|\L\partial_x(uv)\|_{\ZZ_\beta}\lesssim (1+(\beta+\beta^2) T^\nu)\|u\|_{\ZZ_\beta}\|v\|_{\ZZ_\beta}.$$
This ensures that (\ref{est-bilZ}) holds for $T\sim \beta^{-2/\nu}\le 1 $.
\end{proof}
%%%%%%%%%%%%%%%%%%%%%%%%%%%%%%%%%%%%%%%%%%%%%%%
%%%%%%%%%%%%%%%%%%%%%%%%%%%%%%%%%%%%%%%%%%%%%%%%%

\begin{center}Luc Molinet, \\
{\small 
  Laboratoire de Math\'ematiques et Physique Th\'eorique, Universit\'e Fran\c cois Rabelais Tours, F\'ed\'eration Denis Poisson-CNRS, Parc Grandmont, 37200 Tours, FRANCE.  {\it Luc.Molinet@lmpt.univ-tours.fr}}\vspace*{4mm}\\
St\'ephane Vento\\
{\small  L.A.G.A., Institut Galil\'ee, Universit\'e Paris 13,\\
93430 Villetaneuse, FRANCE. {\it vento@math.univ-paris13.fr}}
\end{center}


\begin{thebibliography}{99}
\bibitem{MV2} L. Molinet and S. Vento, {\sl Sharp ill-posedness and well-posedness results for the KdV-Burgers equation: the periodic case} (to be completed).
\end{thebibliography}
\end{document}